\documentclass[12pt]{amsart}

\pdfoutput=1

\usepackage{amsmath,amssymb,amsthm,color}
\usepackage{accents}
\usepackage{hyperref}
\usepackage{url}
\usepackage[top=30truemm,bottom=30truemm,left=25truemm,right=25truemm]{geometry}

\newcommand{\LL}{\mathrm{L}}
\newcommand{\OD}{\mathrm{OD}}
\newcommand{\HOD}{\mathrm{HOD}}
\newcommand{\V}{V}
\newcommand{\PP}{\mathbb{P}}
\newcommand{\QQ}{\mathbb{Q}}
\newcommand{\RR}{\mathbb{R}}
\newcommand{\ZF}{\mathsf{ZF}}
\newcommand{\ZFC}{\mathsf{ZFC}}
\newcommand{\DC}{\mathsf{DC}}
\newcommand{\AD}{\mathsf{AD}}

\newcommand{\Add}{\mathrm{Add}}

\newcommand{\rng}{\mathrm{rng}}
\newcommand{\id}{\mathrm{id}}

\newcommand{\Ord}{\mathrm{Ord}}

\theoremstyle{definition}
\newtheorem{defn}{Definition}[section]
\newtheorem{thm}[defn]{Theorem}
\newtheorem{lem}[defn]{Lemma}

\newtheorem{Q}[defn]{Question}
\newtheorem{claim}{Claim}
\newtheorem{subclaim}{Subclaim}
\newtheorem{case}{Case}

\begin{document}


\title[Preservation of AD via forcings]{Preservation of AD via forcings}

\author[D.\ Ikegami]{Daisuke Ikegami}
\address[D.\ Ikegami]{College of Engineering, Shibaura Institute of Technology, 307 Fukasaku, Minuma-Ward, Saitama City, 337-8570, Saitama JAPAN}

\email[D.\ Ikegami]{\href{mailto:ikegami@shibaura-it.ac.jp}{ikegami@shibaura-it.ac.jp}}

\author[N.\ Trang]{Nam Trang}
\address[N.\ Trang]{Department of Mathematics, University of North Texas, 1155 Union Circle 311430, Denton, TX 76203-5017, USA}

\email[N.\ Trang]{\href{mailto:nam.trang@unt.edu}{nam.trang@unt.edu}}

\thanks{The authors would like to thank W. Hugh Woodin for generously sharing his results and insight on the topic of this paper. They are also grateful to William Chan and Steve Jackson for their work in~\cite{MR4242147} that inspired the work in Section~\ref{sec:the-reals} of this paper. The first author would like to thank the Japan Society for the Promotion of Science (JSPS) for its generous support through the grant with JSPS KAKENHI Grant Number 19K03604. He is also grateful to the Sumitomo Foundation for its generous support through Grant for Basic Science Research. The second author is grateful to the National Science Foundation (NSF) for its generous support through CAREER Grant DMS-1945592.}

\begin{abstract}
We show that assuming $\ZF+\AD^+ + \lq\lq V = \LL \bigl(\wp (\mathbb{R})\bigr)"$, any poset which increases $\Theta$ does not preserve the truth of $\AD$. We also show that in $\ZF + \AD$, any non-trivial poset on $\RR$ does not preserve the truth of $\AD$. This answers the question of Chan and Jackson~\cite[Question~5.7]{MR4242147}. Furthermore, we show that under the assumptions $\ZF + \AD^+ + \lq\lq V = \LL \bigl( \wp (\RR) \bigr)$'' $+ \lq\lq \Theta \text{ is regular}$'', there is a poset on $\Theta$ which adds a new subset of $\Theta$ while preserving the truth of $\AD$. This answers the question of Cunningham~\cite[Section~5]{Cunningham}.
\end{abstract}

\maketitle

\section{Introduction}
In this paper, we discuss the relationship between forcing and the Axiom of Determinacy ($\AD$), especially on the question what kind of forcings preserve the truth of $\AD$. 

Forcing was introduced by Cohen to prove the independence of the Continuum Hypothesis from Zermelo-Fraenkel set theory with the Axiom of Choice ($\ZFC$). Using forcing, he also proved that the Axiom of Choice is independent of Zermelo-Fraenkel set theory ($\ZF$). Since then, forcing has been a basic tool of constructing models of set theory, and it has been used to obtain various results on independence or unprovability of some mathematical statements from set theory as well as to analyze various kinds of models of set theory. 

The Axiom of Determinacy ($\AD$) was introduced by Mycielski and Steinhaus to consider a situation where every set of reals has good properties that simple sets (such as Borel sets and analytic sets) enjoy, and examples of those good properties are Lebesgue measurability and the Baire property. While $\AD$ contradicts the Axiom of Choice in $\ZF$, it has many beautiful consequences on sets of reals. Furthermore, it has been shown that there are deep connections between models of $\ZF + \AD$ (or models of $\ZF + \AD^+$) and models of $\ZFC$ with Woodin cardinals, and $\AD$ has been playing an important role not only in descriptive set theory to analyze the properties of sets of reals, but also in the theory of large cardinals and inner model theory. 

Let us mention how we got interested in the relationship between forcing and $\AD$, especially on the question what kind of forcings preserve the truth of $\AD$. 
By the result of Kunen, there is no non-trivial elementary embedding $j\colon V \to V$ such that $(V, \in , j)$ is a model of $\ZFC$.\footnote{Here $V$ is the class of all sets, $j$ in the structure $(V, \in , j)$ is considered as the interpretation of a binary predicate on the universe, and the structure $(V, \in , j)$ satisfies Comprehension and Replacement for first-order formulas with the binary predicates for $\in$ and $j$.} Furthermore, Hamkins, Kirmayer, and Perlmutter proved that for any set generic $G$ over $V$, there is no non-trivial elementary embedding $j \colon V \to V[G]$ such that $(V [G], \in , j)$ is a model of $\ZFC$. 

One can then ask questions such as what if the structure $(V ,\in , j)$ or $(V[G] , \in , j)$ is a model of $\ZF$ or $\ZF + \AD$ instead of $\ZFC$. 
Using the method of symmetric models, Woodin proved that there are a set generic $G$ over $V$ and a non-trivial elementary embedding $j \colon V \to V[G]$ such that $(V[G], \in , j)$ is a model of $\ZF + \AD$. However, in his example, $j \upharpoonright \Ord$ is the identity map, so there is no critical point of $j$.
 
As far as we know, it is still open whether there are a set generic $G$ over $V$ and an elementary embedding $j \colon V \to V[G]$ such that $(V[G], \in , j)$ is a model of $\ZF + \AD$ and $j \upharpoonright \Ord$ is not the identity map. We are especially interested in the case when the critical point of $j$ is $\omega_1$ in $V$ because if the critical point of $j$ is $\omega_1$ in $V$, then the forcing to obtain $V[G]$ must add new reals to $V$ and $\AD$ has influence on reals and sets of reals. 
To obtain such a $j$, one needs to have a poset $\PP$ to produce such a model $V[G]$, and the poset $\PP$ must add new reals while preserve the truth of $\AD$ from $V$ to $V[G]$. Hence we have a test question: Is there any poset which adds a new real while preserving the truth of $\AD$?
This is how we got interested in the relationship between forcing and $\AD$.

We still do not know if there is any poset which adds a new real while preserving the truth of $\AD$. Considering this question, we have observed that many forcings adding a new real do not preserve the truth of $\AD$. A typical example is Cohen forcing. It is well-known that if $G$ is $V$-generic for Cohen forcing, then in $V[G]$, the set of reals in $V$ does not have the Baire property. In particular, $\AD$ must fail in $V[G]$. 
On the other hand, there are posets which add a new set while preserving the truth of $\AD$. By the result of Woodin~\cite[Section~3]{MR736611}, if we assume $\ZF + \AD + \lq\lq V = \LL (\RR)$'' and let $\kappa$ be a sufficiently big cardinal and $\PP$ be the poset for adding a Cohen subset of $\kappa$ in $\HOD$, the class of hereditarily ordinal definable sets, then the poset $\PP$ adds a new set while preserving the truth of $\AD$. Actually, the poset $\PP$ does not add any set of reals to $\V$ in this case.

We have been wondering what kind of forcings preserve the truth of $\AD$. Our intuition was such a poset $\PP$ would not be able to change the structure of sets of reals drastically. The intuition was partially justified using the ordinal $\Theta$, the supremum of ordinals which are surjective images of $\RR$, by the following theorem:
\newtheorem*{sec3thm}{Theorem~\ref{thm:increasing-Theta-destruction}}
\begin{sec3thm}
Assume $\ZF+\AD^+ + \lq\lq V = \LL \bigl(\wp (\mathbb{R})\bigr)"$. Suppose that a poset $\mathbb{P}$ increases $\Theta$, i.e., $\Theta^{V} < \Theta^{V[G]}$ for any $\mathbb{P}$-generic filter $G$ over $V$. Then $\AD$ fails in $V [G]$ for any $\mathbb{P}$-generic filter $G$ over $V$.
\end{sec3thm}

However, the assumption $\lq\lq V = \LL \bigl(\wp (\mathbb{R})\bigr)"$ is essential in Theorem~\ref{thm:increasing-Theta-destruction}:
\newtheorem*{sec3ex}{Theorem~\ref{thm:increasing-Theta-example}}
\begin{sec3ex}
It is consistent relative to $\mathsf{ZF}+\mathsf{AD}_{\mathbb{R}}$ that $\mathsf{ZF}+\mathsf{AD}$ holds and there is a poset $\mathbb{P}$ increasing $\Theta$ while preserving $\mathsf{AD}$, i.e., for any $\mathbb{P}$-generic filter $G$ over $V$, we have $\Theta^{V} < \Theta^{V[G]}$ and that $\AD$ holds in $V[G]$.\footnote{For an expert on determinacy, the assumption $\ZF + \AD_{\RR}$ is an overkill. The proof of Theorem~\ref{thm:increasing-Theta-example} shows that the assumption $\ZF + \AD^+ + \lq\lq \Theta > \Theta_0$'' is enough.}
\end{sec3ex}
In particular, there is a poset which adds a new set of reals (but does not add a new real) while preserving the truth of $\AD$. 

After we announced Theorem~\ref{thm:increasing-Theta-destruction} and Theorem~\ref{thm:increasing-Theta-example}, Chan and Jackson~\cite{MR4242147} worked on the question what kind of forcings do not preserve the truth of AD. They proved that assuming $\ZF + \AD$, if a non-trivial poset $\PP$ is a wellorderable forcing of cardinality less than $\Theta$, then $\PP$ does not preserve the truth of $\AD$ (\cite[Theorem~3.2]{MR4242147}). They also proved that assuming $\ZF + \AD + \lq\lq \Theta \text{ is regular}$'', if a non-trivial poset $\PP$ is a surjective image of $\RR$, then $\PP$ does not preserve the truth of $\AD$ (\cite[Theorem~5.5]{MR4242147}). 
Then they asked whether $\ZF + \AD$ only (i.e., without assuming the regularity of $\Theta$) implies that if a non-trivial poset $\PP$ is a surjective image of $\RR$, then $\PP$ does not preserve the truth of $\AD$ (\cite[Question~5.7]{MR4242147}).
We give a positive answer to their question:
\newtheorem*{sec4thm}{Theorem~\ref{thm:the-reals}}
\begin{sec4thm}
Assume $\mathsf{ZF}+\mathsf{AD}$. Let $\PP$ be any non-trivial poset which is a surjective image of $\mathbb{R}$ and $G$ be any $\mathbb{P}$-generic filter over $V$. Then $\AD$ fails in $V [G]$.
\end{sec4thm}

We now turn to positive results on the question what kind of forcings preserve the truth of $\AD$. 
As was mentioned in a previous paragraph, By the result of Woodin~\cite[Section~3]{MR736611}, if we assume $\ZF + \AD + \lq\lq V = \LL (\RR)$'' and let $\kappa$ be a sufficiently big cardinal and $\PP$ be the poset for adding a Cohen subset of $\kappa$ in $\HOD$, the class of hereditarily ordinal definable sets, then the poset $\PP$ adds a new set while preserving the truth of $\AD$. A natural question would be how small one can take $\kappa$ for this result. 
Cunningham~\cite{Cunningham} worked on this question. He proved that $\kappa$ can be taken as any regular cardinal larger than $\Theta^{+}$ (\cite[Subsection~4.1]{Cunningham}). Then he asked whether $\kappa$ can be taken as $\Theta$ (\cite[Section~5]{Cunningham}). 
We answer his question positively. In fact, we prove a more general theorem as follows: 
\newtheorem*{sec5thm}{Theorem~\ref{thm:subset-of-Theta-positive}}
\begin{sec5thm}
Assume $\mathsf{ZF}+\mathsf{AD}^+ + \lq\lq V = \mathrm{L} \bigl(\wp (\mathbb{R})\bigr)"$. Suppose that $\Theta$ is regular. Then there is a poset $\mathbb{P}$ on $\Theta$ which adds a subset of $\Theta$ while preserving $\mathsf{AD}$, i.e., for any $\mathbb{P}$-generic filter $G$ over $V$, there is a subset of $\Theta^{V}$ which belongs to $V[G] \setminus V$ and $\AD$ holds in $V[G]$.\footnote{The proof of Theorem~\ref{thm:subset-of-Theta-positive} shows that in both Case~\ref{case:ADRfail} and Case~\ref{case:ADR}, the poset $\PP$ does not add any new set of reals to $V$. In particular, the poset $\PP$ preserves the truth of $\AD^+$ as well.}
\end{sec5thm}

Notice that $\ZF + \AD + \lq\lq V = \LL (\RR)$'' implies the assumptions of Theorem~\ref{thm:subset-of-Theta-positive} including the regularity of $\Theta$. 
Also, in case of $\lq\lq V = \LL (\RR)$'', the poset $\PP$ is the one for adding a Cohen subset of $\Theta$ in $\HOD$ as in Case~\ref{case:ADRfail} in the proof of Theorem~\ref{thm:subset-of-Theta-positive}. 
Therefore, the arguments for Theorem~\ref{thm:subset-of-Theta-positive} answer the question of Cunningham~\cite[Section~5]{Cunningham}.

We also note that Theorem~\ref{thm:subset-of-Theta-positive} is optimal in the following two senses: In one sense, the size of the poset $\PP$ cannot be smaller than $\Theta$. As was mentioned in a previous paragraph, by the result of Chan and Jackson~\cite[Theorem~3.2]{MR4242147}, any wellorderable forcing of cardinality less than $\Theta$ cannot preserve the truth of $\AD$. Also, by Theorem~\ref{thm:the-reals} (or by the result of Chan and Jackson~\cite[Theorem~5.5]{MR4242147} in case $\Theta$ is regular), any poset on $\RR$ (or a surjective image of $\RR$) cannot preserve the truth of $\AD$. 
In the other sense, unless the poset $\PP$ adds a new real, the poset $\PP$ cannot add a new bounded subset of $\Theta$ while preserving the truth of $\AD$. This is because 
if $\PP$ does not add any real and both $V$ and $V[G]$ are models of $\AD$, then 
by the Moschovakis Coding Lemma, $V$ and $V[G]$ have the same bounded subsets of $\Theta$, leading to the situation that the poset $\PP$ cannot add a bounded subset of $\Theta$.

After looking at Theorem~\ref{thm:subset-of-Theta-positive}, it is natural to ask whether the assumption $\lq\lq \Theta \text{ is regular}$'' is essential there. We do not know the answer to this question. However, in case $\Theta$ is singular, we have $\AD_{\RR}$ under the assumptions of Theorem~\ref{thm:subset-of-Theta-positive}. In case $\AD_{\RR}$ holds, which is Case~\ref{case:ADR} in the proof of Theorem~\ref{thm:subset-of-Theta-positive}, the poset $\PP$ in Theorem~\ref{thm:subset-of-Theta-positive} is for adding a Cohen subset of $\Theta$ in $\HOD$. 
We show that this particular poset cannot preserve the truth of $\AD$ if $\Theta$ is singular:
\newtheorem*{sec5ex}{Theorem~\ref{thm:subset-of-Theta-negative}}
\begin{sec5ex}
Assume $\mathsf{ZF}+\mathsf{AD}^+ + \lq\lq V = \mathrm{L} \bigl(\wp (\mathbb{R})\bigr)"$. Suppose that $\Theta$ is singular and let $\mathbb{P}$ be $\Add (\Theta, 1)$ in HOD, where $\Add (\Theta , 1) = \{ p \mid p \colon \gamma \to 2 \text{ for some $\gamma < \Theta$} \}$. Then $\AD$ fails in $V [G]$ for any $\mathbb{P}$-generic filter $G$ over $V$.
\end{sec5ex}


\section{Basic definitions, theorems, and lemmas}\label{sec:basicnotions}

In this section, we introduce basic definitions, theorems, and lemmas we will use in later sections of the paper.
We assume that readers are familiar with the basics of forcing and descriptive set theory. For basic definitions not given in this paper, see Jech~\cite{Jech} and Moschovakis~\cite{new_Moschovakis}.
When we say \lq \lq reals", we mean elements of the Baire space $\omega^{\omega}$ or of the Cantor space $2^{\omega}$.

We start with some basic definitions which will be used throughout the paper:
\begin{defn}

${}$

\begin{enumerate}
\item The ordinal {\it $\Theta$} is the supremum of ordinals which are surjective images of $\RR$. 

\item A set $x$ is {\it OD from sets $y_1 , \ldots , y_n$} if $x$ is definable by a first-order formula with an ordinal and $y_1 , \ldots , y_n$ as parameters.

\item Let $Y$ be a set. We say a set $x$ is {\it hereditarily $\OD_Y$} if any element of the transitive closure of $\{ x \}$ is OD from some elements of $Y$.

\item For a set $Y$, we write {$\HOD_Y$} for the collection of sets which are hereditarily $\OD_Y$. When $Y$ is the empty set, we simply write $\HOD$ for $\HOD_Y$.
\end{enumerate}
\end{defn}

\begin{defn}
Let $A$ and $B$ be sets of reals (or subsets of the Baire space $\omega^{\omega}$). We say $A$ is {\it Wadge reducible to $B$} if there is a continuous function $f \colon \omega^{\omega} \to \omega^{\omega}$ such that $A = f^{-1} (B)$. When $A$ is Wadge reducible to $B$, we write $A \le_{\text{W}} B$. The order $\le_{\text{W}}$ is called the {\it Wadge order on sets of reals}.
\end{defn}

\begin{lem}[Wadge's Lemma]
Assume $\ZF + \AD$. Then for any sets $A , B$ of reals, we have either $A \le_{\text{W}} B$ or $B \le_{\text{W}} \omega^{\omega} \setminus A$.
\end{lem}

\begin{proof}
See e.g., \cite[Lemma~2.1]{Wadge_van_Wesep}.
\end{proof}



The following theorems will be used in Section~\ref{sec:increasing-Theta}.

\begin{thm}[Woodin]\label{thm:HODLA}
Assume $\ZF+\AD^+ + \lq\lq V = \LL \bigl(\wp (\mathbb{R})\bigr)"$. Then the model $\HOD$ is of the form $\LL [Z]$ for some subset $Z$ of $\Theta$, and there are a poset $\QQ$ in $\HOD$ and a $\QQ$-generic filter $H$ over $\HOD$ such that $\HOD \subseteq V \subseteq \HOD[H]$. 
\end{thm}

\begin{proof}
See e.g., \cite[Theorem~3.1.9]{NamThesis}.
\end{proof}

\begin{thm}[Moschovakis]\label{thm:Moschovakis}
Assume $\ZF+\AD$. Then $\Theta$ is a limit of measurable cardinals.
\end{thm}

\begin{proof}
For a proof without assuming $\DC_{\RR}$, one could first prove that $\Theta$ is a limit of strong partition cardinals under $\ZF + \AD$ as in \cite{ADplus} and then verify that every strong partition cardinal is measurable under $\ZF$ as in \cite[28.10~Theorem]{MR1994835}.
\end{proof}

\begin{thm}[Solovay]\label{thm:Solovay}
Assume $\ZF + \AD_{\RR}$. Then for any set $A$ of reals, there is a set $B$ of reals which is not $\OD$ from $A$ and any real.
\end{thm}

\begin{proof}
See \cite[Lemma~2.2]{Solovay_AD_R}.
\end{proof}

The following theorem will be useed in Section~\ref{sec:the-reals}:
\begin{thm}[Chan and Jackson]\label{thm:chan-Jackson}
Assume $\mathsf{ZF}+\mathsf{AD}$ and $\Theta$ is regular. Then for any non-trivial poset $\mathbb{P}$ on $\mathbb{R}$ and any $\mathbb{P}$-generic filter $G$ over $V$, the axiom $\AD$ fails in $V [G]$.
\end{thm}

\begin{proof}
See \cite[Theorem~5.5]{MR4242147}.
\end{proof}

The following theorems will be used in Section~\ref{sec:subset-of-Theta}:
\begin{thm}[Moschovakis]\label{thm:CodingLemma}
Assume $\ZF + \AD$. Then for any non-zero ordinal $\gamma < \Theta$, there is a set $A$ of reals such that there is a surjection from $\RR$ to $\wp (\gamma)$ which is $\OD$ from $A$. 
\end{thm}

\begin{proof}
For any surjection $\rho \colon \RR \to \gamma$, the arguments in \cite[28.15~Theorem]{MR1994835} give us a surjection from $\RR$ to $\wp (\gamma)$ which is OD from $\rho$. If $A$ is a prewellordering on $\RR$ of length $\gamma$, then the surjection $\rho \colon \RR \to \gamma$ induced from $A$ is clearly OD from $A$. Hence there is a surjection from $\RR$ to $\wp (\gamma)$ which is OD from $A$, as desired.
\end{proof}

\begin{thm}[Woodin]\label{thm:AD+ADRfail}
Assume $\ZF+\AD^+ + \lq\lq V = \LL \bigl(\wp (\mathbb{R})\bigr)"$. Suppose also that $\AD_{\RR}$ fails. Then there is a set $T$ of ordinals such that $V = \LL (T, \RR)$.
\end{thm}

\begin{proof}
By the results of Woodin~\cite{ADplus}, the axiom $\AD^+$ and the failure of $\AD_{\RR}$ imply that the set of Suslin cardinals is closed below $\Theta$ while not cofinal in $\Theta$. Hence there is a largest Suslin cardinal in $\Theta$. 
By the result of Woodin~\cite[Corollary~6]{ADplusreflection}, the assumptions $\AD^+$ and $\lq\lq V = \LL \bigl(\wp (\mathbb{R})\bigr)"$ imply that the ultrapower $V^{\mathcal{D}} / \mu$ is well-founded where $\mathcal{D}$ is the set of Turing degrees and $\mu$ is the Martin measure on $\mathcal{D}$. 
Using the result of Woodin~\cite{ADplus}, it follows that there is a set $T$ of ordinals such that $\wp (\RR) \subseteq \LL  (T, \RR)$. Since we assume $V = \LL \bigl( \wp (\RR) \bigr)$, we have $V = \LL  (T, \RR)$, as desired.
\end{proof}

\begin{thm}[Woodin]\label{thm:V=LTR}
Assume $\ZF + \AD^+ + \lq\lq V = \LL (T, \RR)$'' for some set $T$ of ordinals. Then 
\begin{enumerate}
\item for some subset $Z$ of $\Theta$, we have $\HOD_{\{ T \}} = \LL [T, Z]$, and

\item for any real $x$, we have $\HOD_{\{T, x\}} = \HOD_{ \{ T \}} [ x]$.
\end{enumerate}
\end{thm}

\begin{proof}
For (1), one can argue in the same way as in \cite[Corollary~7.21]{MR4132099}. 

For (2), see \cite{ADplus}.
\end{proof}

We next introduce Vop\v{e}nka algebras and their variants we will use in this paper:
\begin{defn}\label{def:Vop}
Let $\gamma$ be a non-zero ordinal and $T$ be a set of ordinals. 
\begin{enumerate}
\item Let $n$ be a natural number with $n \ge 1$ and $\mathcal{O}_n$ be the collection of all nonempty subsets of $(\gamma^{\omega})^n$ which are OD from $T$. Fix a bijection $\pi_n \colon \eta \to \mathcal{O}_n$ which is OD from $T$, where $\eta$ is some ordinal. Let $\QQ_n$ be the poset on $\eta$ such that for each $p , q$ in $\QQ_n$, we have $p \le q$ if $\pi_n (p) \subseteq \pi_n (q)$.  We call $\QQ_n$ the{\it Vop\v{e}nka algebra for adding an element of $(\gamma^{\omega})^n$ in $\HOD_{\{ T \}}$}.

\item For all natural numbers $\ell$ and $m$ with $1 \le \ell \le m$, let $i_{\ell , m} \colon \QQ_{\ell} \to \QQ_m$ be the inclusion map induced from $\pi_{\ell}$ and $\pi_m$, i.e., for all $p \in \QQ_{\ell}$, $\pi_m \bigl( i_{\ell , m} (p)\bigr) = \{ x \in (\gamma^{\omega})^m \mid x \upharpoonright \ell \in \pi_{\ell} (p) \}$. Then each $i_{\ell , m}$ is a complete embedding between posets. Let $\bigl(\QQ_{\omega} , (i_n \colon \QQ_n \to \QQ_{\omega} \mid n < \omega) \bigr)$ be the direct limit of the system $( i_{\ell , m} \colon \QQ_{\ell} \to \QQ_m \mid 1 \le \ell \le m < \omega)$. We call $\QQ_{\omega}$ the {\it finite support direct limit of Vop\v{e}nka algebras for adding an element of $\gamma^{\omega}$ in $\HOD_{\{ T \}}$}.
\end{enumerate}

\end{defn}

The following lemmas will be useful in Section~\ref{sec:subset-of-Theta}:
\begin{lem}\label{lem:Qomega}
Assume $\ZF + \AD^+ + \lq\lq V = \LL (T, \RR)$'' for some set $T$ of ordinals. 
\begin{enumerate}
\item Let $\QQ_1$ be the Vop\v{e}nka algebra for adding an element of $2^{\omega}$ in $\HOD_{\{ T \}}$. Then the poset $\QQ_1$ is of size at most $\Theta$ and $\QQ_1$ has the $\Theta$-c.c. in $\HOD_{\{ T \}}$.

\item Let $\QQ_{\omega}$ be the finite support limit of the Vop\v{e}nka algebras for adding an element of $2^{\omega}$ in $\HOD_{\{ T \}}$.
Then $\QQ_{\omega}$ has the $\Theta$-c.c. in $\HOD_{\{ T \}}$.

\item (Woodin) 
There is a $\QQ_{\omega}$-generic filter $H$ over $\HOD_{\{ T \}}$ such that $ V = \LL (T, \RR) \subseteq  \HOD_{\{ T \}}[H]$ and the set $\RR^V$ is countable in $\HOD_{\{ T \}} [H]$. 
\end{enumerate}
\end{lem}

\begin{proof}
For (1), we first show that the poset $\QQ_1$ is of size at most $\Theta$ in $\HOD_{\{ T \}}$. Recall from Definition~\ref{def:Vop} that $\QQ_1$ is a poset on some ordinal $\eta$ and $\pi_1$ is a bijection from $\eta$ to $\mathcal{O}_1$ which is OD from $T$, where $\mathcal{O}_1$ is the collection of all subsets of $2^{\omega}$ which are OD from $T$. We will argue that the ordinal $\eta$ is at most $\Theta$. For each $\alpha < \Theta$, let $W_{\alpha}$ be the collection of sets of reals in $\mathcal{O}_1$ of Wadge rank $\alpha$. Then we have $\mathcal{O}_1 = \bigcup_{\alpha < \Theta} W_{\alpha}$ and each $W_{\alpha}$ is a surjective image of $\RR$. Since the set $\mathcal{O}_1$ is well-ordered, so is each $W_{\alpha}$ and there is a surjection from $\Theta$ to $W_{\alpha}$ which is OD from $T$. Hence there is a surjection from $\Theta \times \Theta$ to $\mathcal{O}_1$ which is OD form $T$, and therefore the set $\QQ_1$ is of size at most $\Theta$ in $\HOD_{\{ T \}}$, as desired.

We next show that the poset $\QQ_1$ has the $\Theta$-c.c. in $\HOD_{\{ T \}}$. To derive a contradiction, suppose that there is an antichain $(p_{\alpha} \mid \alpha < \Theta)$ in $\QQ_1$ in $\HOD_{ \{ T \}}$. Then the family $\{ \pi_1 (p_{\alpha}) \mid \alpha < \Theta \}$ is a pairwise disjoint family of nonempty subsets of $2^{\omega}$, which would easily induce a surjection from $\RR$ to $\Theta$, contradicting the definition of $\Theta$. Therefore, the poset $\QQ_1$ has the $\Theta$-c.c. in $\HOD_{\{ T \}}$, as desired.

For (2), we first note that for all natural numbers $n$ with $n\ge 1$, the poset $\QQ_n$ has the $\Theta$-c.c. in $\HOD_{\{ T \}}$ by the same argument as in (1). Then using the facts that $\Theta$ is regular in $V = \LL (T, \RR)$ and that $\QQ_{\omega}$ is the direct limit of $\QQ_n$s, it follows that the poset $\QQ_{\omega}$ has the $\Theta$-c.c. in $\HOD_{\{ T \}}$ as well.


For (3), one can argue in the same way as in \cite[Lemma~3.4 and Lemma~3.5]{MR2463615} by replacing $\mathcal{M}$ with $V$, and $\mathcal{H}$ with $\HOD_{\{ T \}}$.
\end{proof}

\begin{lem}\label{lem:useful}
Assume $\mathsf{ZFC}$. Let $\lambda$ be a regular uncountable cardinal, $\mathbb{P}$ be a $<$$\lambda$-closed poset, and $\mathbb{Q}$ be a $\lambda$-c.c. poset. Then for any $\mathbb{P}$-generic filter $G$ over $V$, the poset $\mathbb{Q}$ still has the $\lambda$-c.c. in $V[G]$. Furthermore, if $H$ is a $\QQ$-generic filter over $V$, then $H$ is $\QQ$-generic over $V[G]$ as well.
\end{lem}

\begin{proof}
Let $G$ be a $\mathbb{P}$-generic filter and $A$ be an antichain in $\mathbb{Q}$ in $V[G]$. We will show that $A$ is of size less than $\lambda$ in $V[G]$.

Towards a contradiction, we assume that $A$ is of size at least $\lambda$ in $V[G]$. 

Let $\dot{A}$ be a $\mathbb{P}$-name with $\dot{A}^G = A$. Let $p$ be a condition in $G$ with $p \Vdash_{\mathbb{P}} \lq\lq \dot{A}$ is an antichain in $\check{\mathbb{Q}}$ of size at least $\check{\lambda} "$. Using the $<$$\lambda$-closure of $\mathbb{P}$ in $V$, one can construct a decreasing sequence $(p_{\alpha} \mid \alpha < \lambda)$ in $\mathbb{P}$ and a sequence $(a_{\alpha} \mid \alpha < \lambda)$ in $\mathbb{P}$ with the following properties: 
\begin{enumerate}
\item $p_0 = p$,

\item for all $\alpha , \beta < \lambda$ with $\alpha \neq \beta$, we have $a_{\alpha} \neq a_{\beta}$, and

\item for all $\alpha < \lambda$, $p_{\alpha} \Vdash_{\mathbb{P}} \lq\lq \check{a}_{\alpha} \in \dot{A}"$. 
\end{enumerate}

Since $p \Vdash_{\mathbb{P}} \lq\lq \dot{A} \text{ is an antichain in } \check{\mathbb{Q}} "$, by (1), (2), and (3) above, for all $\alpha , \beta < \lambda$ with $\alpha < \beta$, the condition $p_{\beta}$ forces that $a_{\alpha}$ and $a_{\beta}$ are incompatible in $\mathbb{P}$. Therefore, the set $B = \{ a_{\alpha} \mid \alpha < \lambda \}$ is an antichain in $\mathbb{Q}$ of size $\lambda$ in $V$. This contradicts the assumption that $\mathbb{Q}$ has the $\lambda$-c.c. in $V$. 
Therefore, the antichain $A$ is of size less than $\lambda$ in $V[G]$, as desired.

Let $H$ be a $\QQ$-generic filter over $V$. We will verify that $H$ is $\QQ$-generic over $V[G]$ as well. Let $A$ be a maximal antichain in $\QQ$ in $V[G]$. We will see that $H \cap A \neq \emptyset$. By the arguments in the previous paragraphs, $A$ is of size less than $\lambda$ in $V[G]$. Since $G$ is $\PP$-generic over $V$ and $\PP$ is $<$$\lambda$-closed in $V$ while $\QQ$ is in $V$, there is no subset of $\QQ$ of size less than $\lambda$ in $V[G] \setminus V$. Hence the antichain $A$ is in $V$ as well. By the genericity of $H$ over $V$, we have that $H \cap A \neq \emptyset$, as desired.
\end{proof}

\begin{lem}\label{thm:ADR-countable-sequences}
Assume $\ZF + \AD^+ + \AD_{\RR}$. 
Then for any set $C$ of reals, there is an $s \in \Theta^{\omega}$ such that $C$ is OD from $s$ and that $C$ is in $\HOD_{\{ s\}} (\RR)$.
\end{lem}

\begin{proof}
Let $C$ be any set of reals. By the result of Woodin~\cite{ADplus}, under $\ZF + \AD^+ + \AD_{\RR}$, every set of reals is Suslin. By the result of Martin and Steel~\cite{MR2463620}, every Suslin and co-Suslin set of reals is homogeneously Suslin. In particular, the complement $2^{\omega} \setminus C$ is homogeneously Suslin witnessed by the sequence $(\mu_u \mid u \in 2^{<\omega} )$ of measures on $\kappa^{<\omega}$ for some $\kappa < \Theta$. By the result of Kunen~\cite[28.21~Corollary]{MR1994835}, each measure $\mu_u$ is OD. Using the Moschovakis Coding Lemma and $\AD_{\RR}$, one can show that each measure $\mu_u$ is definable from an ordinal below $\Theta$. Hence there is an $s \in \Theta^{\omega}$ such that the sequence $(\mu_u \mid u \in 2^{<\omega} )$ is definable from $s$. 
Now from the sequence $(\mu_u \mid u \in 2^{<\omega} )$, one can construct a Martin-Solovay tree $T$ such that $C = \text{p} [T]$. By the construction of $T$, it follows that $T$ is OD from $(\mu_u \mid u \in 2^{<\omega} )$. Hence $T$ is OD from $s$, which easily implies that the set $C$ is OD from $s$ and $C$ is in $\HOD_{\{ s\}} (\RR)$, as desired.
\end{proof}

\begin{lem}\label{lem:ADR-countable-sequences}
Assume $\ZF + \AD_{\RR} $. 
Let $\gamma < \Theta$ and $\QQ_1$ be the Vop\v{e}nka algebra for adding an element of $\gamma^{\omega}$ in $\HOD$. 
Also let $\QQ_{\omega}$ be the finite support limit of the Vop\v{e}nka algebras for adding an element of $\gamma^{\omega}$ in $\HOD$. 
\begin{enumerate}
\item The posets $\QQ_1$ and $\QQ_{\omega}$ are of size less than $\Theta$ in $\HOD$.

\item Let $s \in \gamma^{\omega}$ and $h_s = \{ p \in \QQ_1 \mid s \in \pi_1 (p) \}$, where $\pi_1 \colon \QQ_1 \to \mathcal{O}_1$ is as in Definition~\ref{def:Vop}. Then the set $h_s$ is a $\QQ_1$-generic filter over $\HOD$ such that $\HOD [h_s] = \HOD_{ \{ s \}}$.

\item (Woodin) 
There is a $\QQ_{\omega}$-generic filter $H$ over $\HOD$ such that the set $(\gamma^{\omega})^V$ is countable in $\HOD [H]$. 
\end{enumerate}
\end{lem}

\begin{proof}
For (1), we first show that the poset $\QQ_1$ is of size less than $\Theta$ in $\HOD$. Recall from Definition~\ref{def:Vop} that $\pi_1 \colon \QQ_1 \to \mathcal{O}_1$ is a surjection which is OD, where $\mathcal{O}_1$ is the collection of all subsets of $\gamma^{\omega}$ which are OD. Since $\gamma < \Theta$, by Theorem~\ref{thm:CodingLemma}, there is a set $A$ of reals such that there is a surjection from $\RR$ to $\wp (\gamma)$ which is OD from $A$. In particular, there is a surjection $\sigma \colon \RR \to \gamma^{\omega}$ which is OD from $A$. 
Hence for each $b \in \mathcal{O}_1$, the set $\sigma^{-1} (b)$ of reals is OD from $A$. Since we assume $\AD_{\RR}$, by Theorem~\ref{thm:Solovay}, there is a set $B$ of reals which is not OD from $A$. By Wadge's Lemma under $\ZF + \AD$, for each $b \in \mathcal{O}_1$, the set $\sigma^{-1} (b)$ is Wadge reducible to $B$. In particular, there is a surjection from $\RR$ to the family $\{ \sigma^{-1} (b) \mid b \in \mathcal{O}_1 \}$. Hence the family $\mathcal{O}_1$ is also a surjective image of $\RR$ and the poset $\QQ_1$ is of size less than $\Theta$ in $V$. Since $\Theta$ is a cardinal in $V$, it follows that the poset $\QQ_1$ is of size less than $\Theta$ in $\HOD$ as well. 

We next show that the poset $\QQ_{\omega}$ is of size less than $\Theta$ in $\HOD$. Let $C = A \oplus B = \{ x \ast y \mid x \in A \text{ and } y \in B \}$, where $x \ast y (2\ell) = x(\ell)$ and $x \ast y (2\ell+1) = y (\ell)$ for all $\ell < \omega$. Then the argument in the last paragraph shows that there is a surjection from $\RR$ to $\QQ_1$ which is OD from $C$. Similarly, one can argue that for each natural numbers $n$ with $n \ge 1$, there is a surjection from $\RR$ to $\QQ_n$ which is OD from $C$. Since all such surjections are OD from $C$, one can pick a sequence $(\rho_n \colon \RR \to \mathcal{O}_n \mid n \ge 1)$ of surjections, which would readily give us a surjection from $\RR$ to $\QQ_{\omega}$. Therefore, the poset $\QQ_{\omega}$ is of size less than $\Theta$ in $V$. Since $\Theta$ is a cardinal in $V$, it follows that the poset $\QQ_{\omega}$ is of size less than $\Theta$ in $\HOD$ as well. 

For (2), for the $\QQ_1$-genericity of $h_s$ over $\HOD$, see e.g., \cite[Theorem~15.46]{Jech}. 

We will show the equality $\HOD[h_s] = \HOD_{\{s\}}$. The inclusion $\HOD [h_s] \subseteq \HOD_{\{ s \}}$ is easy because $h_s$ is OD from $s$ and $h_s$ is a set of ordinals. We will argue that $\HOD_{\{s \}} \subseteq \HOD [h_s]$. Since $\HOD_{\{ s \}}$ is a model of $\ZFC$, it is enough to see that every set of ordinals in $\HOD_{\{ s \} }$ is also in $\HOD [h_s]$. Let $X$ be any set of ordinals in $\HOD_{\{s \}}$. We will verify that $X$ is also in $\HOD[h_s]$. Let $\delta$ be an ordinal such that $X \subseteq \delta$.
Since $X$ is in $\HOD_{\{s \} }$, the set $X$ is OD from $s$. So there is a formula $\phi$ such that for all $\alpha < \delta$, we have that $\alpha \in X$ if and only if $\phi [\alpha , s]$ holds. 
For each $\alpha < \delta$, let $b_{\alpha} = \{ x \in \gamma^{\omega} \mid \phi [\alpha , x]\}$. Then each set $b_{\alpha}$ is a subset of $\gamma^{\omega}$ which is OD. So each $b_{\alpha}$ is in $\mathcal{O}_1$. Now we have the following equivalences: For all $\alpha < \delta$,
\begin{align*}
\alpha \in X \iff \phi [\alpha , s] \iff s \in b_{\alpha} \iff \pi_1^{-1} (b_{\alpha}) \in h_s.
\end{align*}
Hence $X = \{ \alpha < \delta \mid \pi_1^{-1} (b_{\alpha}) \in h_s \}$. Since the sequence $\bigl(\pi_1^{-1} (b_{\alpha} ) \in \QQ_1 \mid \alpha < \delta \bigr)$ is OD and $\QQ_1$ is in $\HOD$, the sequence $\bigl(\pi_1^{-1} (b_{\alpha} ) \in \QQ_1 \mid \alpha < \delta \bigr)$ belongs to $\HOD$. Hence the set $ \{ \alpha < \delta \mid \pi_1^{-1} (b_{\alpha}) \in h_s \}$ is in $\HOD [h_s]$. Therefore, the set $X$ is in $\HOD [h_s]$, as desired.

For (3), one can argue in the same way as in \cite[Lemma~3.4 and Lemma~3.5]{MR2463615} by replacing $\RR$ with $\gamma^{\omega}$, $\mathcal{M}$ with $V$, and $\mathcal{H}$ with $\HOD$.
\end{proof}


\section{On forcings increasing $\Theta$}\label{sec:increasing-Theta}

In this section, we prove the following theorems:
\begin{thm}\label{thm:increasing-Theta-destruction}
Assume $\ZF+\AD^+ + \lq\lq V = \LL \bigl(\wp (\mathbb{R})\bigr)"$. Suppose that a poset $\mathbb{P}$ increases $\Theta$, i.e., $\Theta^{V} < \Theta^{V[G]}$ for any $\mathbb{P}$-generic filter $G$ over $V$. Then $\AD$ fails in $V [G]$ for any $\mathbb{P}$-generic filter $G$ over $V$.
\end{thm}

\begin{thm}\label{thm:increasing-Theta-example}
It is consistent relative to $\mathsf{ZF}+\mathsf{AD}_{\mathbb{R}}$ that $\mathsf{ZF}+\mathsf{AD}$ holds and there is a poset $\mathbb{P}$ increasing $\Theta$ while preserving $\mathsf{AD}$, i.e., for any $\mathbb{P}$-generic filter $G$ over $V$, we have $\Theta^{V} < \Theta^{V[G]}$ and that $\AD$ holds in $V[G]$.
\end{thm}

\begin{proof}[Proof of Theorem~\ref{thm:increasing-Theta-destruction}]

Let $G$ be a $\PP$-generic filter over $V$. We will show that $\AD$ fails in $V[G]$.
Towards a contradiction, we assume that $\AD$ holds in $V[G]$.

Since we have $\AD^+$ and $V = \LL \bigl (\wp (\RR) \bigr)$, by Theorem~\ref{thm:HODLA}, the model HOD is of the form $\LL [Z]$ for some subset $Z$ of $\Theta$, and there are a poset $\QQ$ in HOD and a $\QQ$-generic filter $H$ over HOD such that $\HOD \subseteq V \subseteq \HOD[H]$. In particular, $Z^{\#}$ does not exist in HOD. Since any poset does not add $Z^{\#}$, it follows that $Z^{\#}$ does not exist in $V$ either.

We will argue that $Z^{\#}$ exists in $V[G]$, which would contradict the fact that $Z^{\#}$ does not exist in $V$.
Since $\PP$ increases $\Theta$, we have $\Theta^V < \Theta^{V[G]}$. 
By assumption, we have $\AD$ in $V[G]$, so by Theorem~\ref{thm:Moschovakis}, it follows that $\Theta^{V[G]}$ is a limit of measurable cardinals in $V[G]$. In particular, there is a measurable cardinal $\kappa$ in $V[G]$ such that $\Theta^V < \kappa$. 
Let $U$ be a $<$$\kappa$-complete nonprincipal ultrafilter on $\kappa$ in $V[G]$. 
Then letting $M = \LL [U, Z]$, the cardinal $\kappa$ is measurable also in $M$ witnessed by $U \cap M$. 
Since $M$ is a model of $\ZFC$ and $Z$ is a bounded subset of $\kappa$ in $M$, it follows that $Z^{\#}$ exists in $M$.
By absolutness of $Z^{\#}$, we have $Z^{\#}$ in $V[G]$, contradicting the fact that $Z^{\#}$ does not exist in $V$.

Therefore, the assumption that $\AD$ holds in $V[G]$ was wrong, and $\AD$ fails in $V[G]$.

This completes the proof of Theorem~\ref{thm:increasing-Theta-destruction}.
\end{proof}

\begin{proof}[Proof of Theorem~\ref{thm:increasing-Theta-example}]

We assume $\ZF + \AD_{\RR}$ and will show that there is an inner model $M$ of $\ZF + \AD$ satisfying that there is a poset $\PP$ increasing $\Theta$ while preserving $\AD$.


Let $M =\HOD_{\RR}$, the class of all sets hereditarily ordinal definable from some real. We will show that $M$ is the desired inner model.

First notice that $M$ is a model of $\ZF$. Also since $M$ contains all the reals and $V$ satisfies $\AD$, we have that $M$ is a model of $\AD$ as well. 
Since we have $\AD_{\RR}$ in $V$, by Theorem~\ref{thm:Solovay}, there is a set $B$ of reals which is not definable from any ordinal and any real. Hence the set $B$ is not in $M$.

We will show that $M$ satisfies that there is a poset $\PP$ increasing $\Theta$ while preserving $\AD$. The idea is to consider a variant of Vop\v{e}nka algebra in $M$ adding the set $B$ to $M$. 
Let $\mathcal{O} = \{ b \subseteq \wp (\RR) \mid \text{$b$ is nonempty and OD from some real}\}$ ordered by inclusion. Then $\mathcal{O}$ is a poset which is OD.
Let $\eta$ be a sufficiently large ordinal and let $\pi \colon \eta \times \RR \to \mathcal{O}$ be a surjection which is OD such that if a set $b$ is in $\mathcal{O}$ and OD from a real $x$, then there is some $\alpha < \eta$ such that $\pi (\alpha , x) = b$.
Let $\PP = \pi^{-1} (\mathcal{O})$ and for $p_1, p_2\in \PP$, we set $p_1 \le p_2$ if $\pi (p_1) \subseteq \pi (p_2)$. 
Then since $\pi$ is OD, the poset $\PP$ is in $M$. For an $r $ in $\PP$, let $\PP \upharpoonright r = \{ p \in \PP \mid p \le r\}$. 

We will show that there is some $\PP$-generic filter $G$ over $M$ such that $\Theta^M < \Theta^{M[G]}$ and $M[G]$ is a model of $\AD$. This is enough to end the arguments for the theorem because then there is some $r \in \PP$ forcing the desired two statements for $M[G]$ over $M$, and the poset $\PP \upharpoonright r$ is the desired poset in $M$.

Let $H = \{ b \in \mathcal{O} \mid B \in b\}$ and $G = \pi^{-1} (H)$. We will see that $G$ is the desired filter.

We first verify that $G$ is $\PP$-generic over $M$. Let $D$ be a dense subset of $\PP$ in $M$. We will argue that $G \cap D \neq \emptyset$. 
Let $E = \pi [D]$ and $b_E = \bigcup E$. By the definition of $\PP$, the set $E$ is dense in $\mathcal{O}$. We claim that $b_E = \wp ( \RR )$. Suppose not. Then since $D$ is in $M$ and $\pi$ is OD, the set $b_E$ is OD from some real. So $b_E$ is in $\mathcal{O}$. But then $\wp (\RR) \setminus b_E$ is a nonempty set which is in $\mathcal{O}$ incompatible with any element of $E$, contradicting that $E$ is dense in $\mathcal{O}$. 
Hence $b_E = \wp (\RR)$. Since $B$ is in $\wp (\RR)$, we have that $B$ is in $b_E$, so there is a $b'$ in $E$ such that $B$ is in $b'$. By the definition of $H$, the condition $b'$ is also in $H$. Hence $H \cap E \neq \emptyset$. Since $G = \pi^{-1} (H)$ and $E = \pi [D]$, it follows that $G \cap D \neq \emptyset$, as desired. Therefore, $G$ is $\PP$-generic over $M$.

We next verify that $M[G]$ is a model of $\AD$. Since $B$ is in $V$, $H = \{ b \in \mathcal{O} \mid B \in b\}$, and $G = \pi^{-1} (H)$, it follows that $G$ is in $V$ and $M[G]$ is a submodel of $V$. Since $M$ contains all the reals, so does $M[G]$. Finally, since $V$ is a model of $\AD$, it follows that $M[G]$ is also a model of $\AD$, as desired.

Finally, we verify that $\Theta^M < \Theta^{M[G]}$. Since both $M$ and $M[G]$ are models of $\AD$ containing all the reals, by the Wadge lemma under $\ZF + \AD$, it is enough to see that there is a set of reals in $M[G] \setminus M$. Since $B$ is not in $M$, it suffices to argue that $B$ is in $M[G]$. 
For each real $x$, let $b_x = \{ A \in \wp (\RR) \mid x \in A\}$. Then $b_x$ is OD from $x$, so $b_x$ is in $\mathcal{O}$. 
By the choice of $\pi$, for each real $x$, there is an ordinal $\alpha$ such that $\pi (\alpha , x) = b_x$. For each real $x$, let $\alpha_x$ be the least ordinal with $\pi (\alpha_x , x) = b_x$. Then since $\pi$ and $\mathcal{O}$ are OD, the sequence $( \alpha_x \mid x \in \RR)$ is OD and is in $M = \HOD_{\RR}$. 
From the sequence $(\alpha_x \mid x \in \RR)$ and $G$, one can compute the set $B$ as follows: for any real $x$,
\begin{align*}
x \in B \iff B \in b_x \iff b_x \in H \iff (\alpha_x , x) \in G.
\end{align*} 
Therefore, the set $B$ is in $M[G]$, as desired.

We have verified that $G$ is the desired filter, and this completes the proof of Theorem~\ref{thm:increasing-Theta-example}.
\end{proof}

\section{On forcings on the reals}\label{sec:the-reals}

In this section, we prove the following theorem which answers a question by Chan and Jackson~\cite[Question~5.7]{MR4242147}:
\begin{thm}\label{thm:the-reals}
Assume $\mathsf{ZF}+\mathsf{AD}$. Let $\PP$ be any non-trivial poset which is a surjective image of $\mathbb{R}$ and $G$ be any $\mathbb{P}$-generic filter over $V$. Then $\AD$ fails in $V [G]$.
\end{thm}

\begin{proof}[Proof of Theorem~\ref{thm:the-reals}]

Let $\PP$ be any non-trivial poset which is a surjective image of $\RR$ and $G$ be any $\PP$-generic filter over $V$. We will show that $\AD$ fails in $V[G]$.
Since $\PP$ is a surjective image of $\RR$, there is a poset on $\RR$ which is forcing equivalent to $\PP$. Hence we may assume $\PP$ is a poset on $\RR$.

Towards a contraction, we assume that $\AD$ holds in $V[G]$.


\begin{case}\label{RV-uncountable}
When the set $\RR^V$ is uncountable in $V[G]$.
\end{case}

Here is the key point:
\begin{claim}\label{key-claim}
There is a real $r_0$ in $V[G]$ such that $\RR^{V[G]} \subseteq \LL (\RR^V , r_0)$.
\end{claim}

\begin{proof}[Proof of Claim~\ref{key-claim}]
Since $V[G]$ satisfies $\AD$, the set $\RR^V$ has the perfect set property in $V[G]$. 
Since $\RR^V$ is uncountable in $V[G]$, the set $\RR^V$ contains a perfect set $C$ in $V[G]$. 
Let $r_0$ code a perfect tree $T$ on $2 = \{0,1\}$ with $[T] = C$ in $V[G]$.

We will show that $\RR^{V[G]} \subseteq \LL (\RR^V , r_0)$.
Let $x$ be any element of $2^{\omega}$ in $V[G]$. We will see that $x$ is in $\LL ( \RR^V , r_0)$. 

We say a node $t \in T$ is {\it splitting in $T$} if both $t^{\frown} \langle 0\rangle$ and $t^{\frown} \langle 1 \rangle $ are in $T$. 
Let $\{ t_s \in T \mid s \in 2^{<\omega} \}$ be the set of all splitting nodes in $T$ such that if $s_1 $ is a subsequence of $s_2$ in $2^{<\omega}$, then $t_{s_1}$ is a subsequence of $t_{s_2}$ in $T$.
Let $y = \bigcup \{ t_{x \upharpoonright n} \mid n < \omega \}$. Then $y$ is in $[T]$. Since $[T] =C \subseteq \RR^V$, the real $y$ is in $\RR^V$. 
However, for all $n< \omega$ and $k \in 2 = \{ 0,1 \}$, 
\begin{align*}
x(n) = k \iff t_{(x\upharpoonright n)^{\frown} \langle k \rangle} \subseteq y.
\end{align*}
Hence $x$ can be simply computed from $y$ and $T$. So $x \in \LL [y, T]$. Since $\LL [y, T] \subseteq \LL [y, r_0] \subseteq \LL (\RR^V, r_0)$, the real $x$ is in $\LL (\RR^V , r_0)$, as desired.

This completes the proof of Claim~\ref{key-claim}.
\end{proof}

Continuing to argue in Case~\ref{RV-uncountable}, let $r_0$ be a real in $V[G]$ such that $\RR^{V[G]} \subseteq \LL (\RR^V , r_0)$ as in Claim~\ref{key-claim}.
Since the poset $\PP$ is on $\RR^V$, there is a $\PP$-name $\dot{x}$ such that $\dot{x}^G = r_0$ and $\dot{x}$ is coded by some set $A$ of reals in $V$. 
Then setting $M = \LL (\RR^V , \PP , A)$, we have that $M$ is an inner model of $V$ satisfying $\AD$ and the statement \lq\lq $\Theta$ is regular''. 
However, since $\RR^{V[G]} \subseteq \LL (\RR^V , r_0) \subseteq M[G] \subset V[G]$ and we assumed that $V[G]$ satisfies $\AD$, the model $M[G]$ also satisfies $\AD$, contradicting Theorem~\ref{thm:chan-Jackson}.
Therefore, the assumption that $V[G]$ satisfies $\AD$ was wrong and $\AD$ must fail in $V[G]$, as desired.

This finsihes the arguments for Theorem~\ref{thm:the-reals} in Case~\ref{RV-uncountable}.

\begin{case}\label{RV-countable}
When the set $\RR^V$ is countable in $V[G]$.
\end{case}

Since $\RR^V$ is countable in $V[G]$, any ordinal $\alpha$ below $\Theta^V$ is countable in $V[G]$ as well. Hence $\Theta^V \le \omega_1^{V[G]}$.

We will show that $\Theta^{V[G]} \le (\Theta^+)^V$, which would contradict the assumption that $\AD$ holds in $V[G]$, because $\AD$ in $V[G]$ would imply that $\Theta^{V[G]} >  \omega_2^{V[G]} \ge (\Theta^+)^V$ since $\Theta^V \le \omega_1^{V[G]}$.

To see that $\Theta^{V[G]} \le (\Theta^+)^V$, let $f\colon \RR^{V[G]} \to (\Theta^+)^V$ be any function in $V[G]$. We will show that $f$ is not surjective. 
As in the arguments in Case~\ref{RV-uncountable}, since $\PP$ is on $\RR$, any real in $V[G]$ can be coded by a set of reals in $V$. Hence we may assume that $f\colon \wp (\RR)^V \to (\Theta^+)^V$. 
Also, since $\PP$ is on $\RR^V$, there is a function $g \colon \wp (\RR)^V \times \RR^V \to (\Theta^+)^V$ in $V$ such that $\rng (f) \subseteq \rng (g)$. Therefore, it is enough to see that $g$ is not surjective in $V$.

We now work in $V$. To see that $g$ is not surjective, for each $\alpha < \Theta$, let $W_{\alpha}  = \{ B \in \wp (\RR) \mid |B|_{\text{W}} = \alpha \}$, where $|B|_{\text{W}}$ is the Wadge ordinal of $B$. Then each $W_{\alpha}$ is a surjective image of $\RR$ and so is the set $R_{\alpha} = \{ g(B, x) \mid B \in W_{\alpha}, x\in \RR \}$. 
Hence, for every $\alpha < \Theta$, the order type of $R_{\alpha }$ is less than $\Theta$, and $\rng (g) = \bigcup_{\alpha < \Theta} R_{\alpha}$ is a surjective image of $\Theta \times \Theta$. 
Therefore, $\rng (g)$ is of cardinality at most $\Theta$ which is smaller than $\Theta^+$. Hence $g$ is not surjective in $V$, as desired.

This finishes the arguments for Theorem~\ref{thm:the-reals} in Case~\ref{RV-countable}.

This completes the proof of Theorem~\ref{thm:the-reals}.
\end{proof}

\section{On forcings adding a subset of $\Theta$}\label{sec:subset-of-Theta}

In this section, we prove the following theorems:
\begin{thm}\label{thm:subset-of-Theta-positive}
Assume $\mathsf{ZF}+\mathsf{AD}^+ + \lq\lq V = \mathrm{L} \bigl(\wp (\mathbb{R})\bigr)"$. Suppose that $\Theta$ is regular. Then there is a poset $\mathbb{P}$ on $\Theta$ which adds a subset of $\Theta$ while presering $\mathsf{AD}$, i.e., for any $\mathbb{P}$-generic filter $G$ over $V$, there is a subset of $\Theta^{V}$ which belongs to $V[G] \setminus V$ and $\AD$ holds in $V[G]$.
\end{thm}

\begin{thm}\label{thm:subset-of-Theta-negative}
Assume $\mathsf{ZF}+\mathsf{AD}^+ + \lq\lq V = \mathrm{L} \bigl(\wp (\mathbb{R})\bigr)"$. Suppose that $\Theta$ is singular and let $\mathbb{P}$ be $\Add (\Theta, 1)$ in HOD, where $\Add (\Theta , 1) = \{ p \mid p \colon \gamma \to 2 \text{ for some $\gamma < \Theta$} \}$. Then $\AD$ fails in $V [G]$ for any $\mathbb{P}$-generic filter $G$ over $V$.
\end{thm}

\begin{proof}[Proof of Theorem~\ref{thm:subset-of-Theta-positive}]

Throughout the proof of the theorem, we write $\lambda$ for $\Theta^V$.

We prove the theorem by considering the two cases whether $\mathsf{AD}_{\mathbb{R}}$ holds or not.

 \setcounter{case}{0}
\begin{case}\label{case:ADRfail}
When $\mathsf{AD}_{\mathbb{R}}$ fails.
\end{case}

Since $\mathsf{AD}_{\mathbb{R}}$ fails while we assume $\mathsf{AD}^+$ and $V = \mathrm{L}\bigl( \wp(\mathbb{R}) \bigr)$, by Theorem~\ref{thm:AD+ADRfail}, there is a set $T$ of ordinals such that $V = \mathrm{L}(T, \mathbb{R})$. We fix such a $T$ throughout the arguments for Case~\ref{case:ADRfail}.

Let $\mathbb{P}$ be $\Add (\lambda, 1)$ in $\HOD_{\{ T \}}$, where $\Add (\lambda , 1) = \{ p \mid p \colon \gamma \to 2 = \{ 0, 1 \} \text{ for some $\gamma < \lambda$} \}$. 
Since $\mathbb{P}$ is computed in $\HOD_{\{ T \}}$ and $\lambda = \Theta^V$ is inaccessible in $\HOD_{\{ T \}}$, the poset $\mathbb{P}$ can be considered as a poset on $\lambda$. 

We will show that $\PP$ is the desired poset in Case~\ref{case:ADRfail}, i.e., $\PP$ adds a subset of $\lambda = \Theta^V$ while preserving $\AD$ in Case~\ref{case:ADRfail}.

Let $G$ be any $\PP$-generic filter over $V$. Then the function $\bigcup G \colon \lambda \to 2$ can be considered as a subset of $\lambda$ and by the genericity of $G$ over $V$, the subset is not in $V$. Hence the poset $\PP$ adds a new subset of $\lambda$ to $V$.

We will show that $\AD$ holds in $V[G]$. We start with showing that the poset $\PP$ does not add any bounded subset of $\lambda$:
\setcounter{claim}{0}
\begin{claim}\label{no-bounded-subset-in-vg-non-adr}
For any $\gamma < \lambda$, we have $\wp (\gamma)^V = \wp (\gamma)^{V[G]}$. In particular, $\RR^V = \RR^{V[G]}$.
\end{claim}

\begin{proof}[Proof of Claim~\ref{no-bounded-subset-in-vg-non-adr}]

Let $\gamma$ be an ordinal less than $\lambda$ and $f \colon \gamma \to 2$ in $V[G]$. We will show that $f$ is in $V$. 

Let $\dot{f}$ be a $\PP$-name with $\dot{f}^G = f$. Since $\PP$ can be seen as a poset on $\lambda$ and $f \colon \gamma \to 2$, we may asssume that $\dot{f}$ is a subset of $\lambda \times \gamma \times 2$. 
Since $V = \LL (T, \RR)$ and $\dot{f}$ is in $V$, we have that $\dot{f}$ is $\OD_{\{T,x\}}$ for some real $x$. Then since $\dot{f}$ is essentially a set of ordinals, $\dot{f}$ is in $\HOD_{\{T,x\}}$. 
By Theorem~\ref{thm:V=LTR}, we have that $\HOD_{\{T,x\}} = \HOD_{\{ T \}} [x]$ and $\HOD_{\{ T \}} = \LL [T, Z]$ for some subset $Z$ of $\lambda$. 
Since $f = \dot{f}^G$, it follows that $f$ is in $\HOD_{\{T,x\}} [G] = \HOD_{\{ T \}} [x][G]$. 

Let $\QQ_1$ be the Vop\v{e}nka algebra for adding an element of $2^{\omega}$ in $\HOD_{\{ T \}}$. Then the real $x$ induces a $\QQ_1$-generic filter $h_x$ over $\HOD_{\{ T \}}$ such that $x \in \HOD_{\{ T \}}[h_x]$. 
Since $G$ was chosen to be $\PP$-generic over $V$, it is also $\PP$-generic over $\HOD_{\{ T \}} [h_x]$. Hence the filter $G \times h_x$ is $\PP \times \QQ_1$-generic over $\HOD_{\{ T \}}$ and $\HOD_{\{ T \}} [x][G]  \subseteq HOD_{\{ T \}} [h_x][G]  = \HOD_{\{ T \}} [G] [h_x]$. 


Since $\dot{f}$ is in $\HOD_{\{ T \}} [x] \subseteq \HOD_{\{ T \}} [h_x]$, there is a $\QQ_1$-name $\tau$ in $\HOD_{\{ T \}}$ such that $\tau^{h_x} = \dot{f}$. 
Let $\nu$ be a sufficiently big cardinal in $\HOD_{\{ T \}} [G]$ and let $N$ be $V_{\nu}$ in $\HOD_{\{ T \}} [G]$. 
By the $<$$\lambda$-closure of $\PP$ in $\HOD_{\{ T \}}$, the ordinal $\lambda$ is regular in $\HOD_{\{ T \}}[G]$. 
Since $\HOD_{\{ T \}}[G]$ is a model of $\ZFC$, there is an elementary substructure $X$ of $N$ in $\HOD_{\{ T \}} [G]$ such that $\gamma + 1 \subseteq X$, $X \cap \lambda \in \lambda$, $X$ is of size less than $\lambda$, and $T, Z, G, \PP, \QQ_1 , \tau  \in X$. 
Let $M$ be the transitive collapse of $X$ and let $\pi \colon M \to X$ be the inverse of the collapsing map. 
Then letting $\kappa = X \cap \lambda$, we have that $\kappa$ is the critical point of $\pi$ and $\pi (\kappa) = \lambda$. 
For any $a \in X$, we write $\bar{a}$ for $\pi^{-1} (a)$, i.e., $\pi (\bar{a}) = a$. 

We claim that $M$ is in $\HOD_{\{ T \}}$. Let $g = \bigcup G$. Then by the genericity of $G$, we have $g \colon \lambda \to 2$. Since $G$ is simply definable from $g$, we have $\HOD_{\{ T \}} [G] = \HOD_{\{ T \}} [g]$. 
Recall that $\HOD_{\{ T \}} = \LL [T, Z]$, so $\HOD_{\{ T \}} [G] = \LL [T, Z][G] = \LL [T, Z][g]$. Hence the model $M$ is of the form $\LL_{\mu} [\bar{T}, \bar{Z}][\bar{g}]$ for some $\mu$. 
Since $Z$ is a subset of $\lambda$, we have $\bar{Z} = Z \cap \kappa$ and hence $\bar{Z} \in \HOD_{\{ T \}}$. 
Since $\bar{T}$ is a set of ordinals of size less than $\lambda$ in $\HOD_{\{ T \}} [G]$, by the $<$$\lambda$-closure of $\PP$ in $\HOD_{\{ T \}}$, the set $\bar{T}$ is in $\HOD_{\{ T \}}$.  
Since $g \colon \lambda \to 2$, we have $\bar{g} = g \upharpoonright \mu$, which is in $\PP$. So $\bar{g}$ is in $\HOD_{\{ T\}}$. 
Since $M = \LL_{\mu} [\bar{T}, \bar{Z} , \bar{g}]$, the model $M$ is in $\HOD_{\{ T \}}$, as desired.

Let $\bar{h}_x = \{ \bar{q} \mid q \in h_x \cap X\}$. We claim that $\bar{h}_x$ is $\bar{\QQ}_1$-generic over $M$. 
Recall that $\QQ_1$ is the Vop\v{e}nka algebra for adding an element of $2^{\omega}$ in $\HOD_{\{ T \}}$. By Lemma~\ref{lem:Qomega}, we may assume that $\QQ_1$ is on $\Theta^V = \lambda$ and $\QQ_1$ has the $\lambda$-c.c. in $\HOD_{\{T\}}$. 
Let $A$ be a maximal antichain in $\bar{\QQ}_1$ such that $A$ is in $M$. We will verify that $A \cap \bar{h}_x \neq \emptyset$. 
Since $\PP$ is $<$$\lambda$-closed and $\QQ_1$ has the $\lambda$-c.c. in $\HOD_{\{ T \}}$, by Lemma~\ref{lem:useful}, the poset $\QQ_1$ still has the $\lambda$-c.c. in $\HOD_{\{ T \}} [G]$. 
By elementarity of $\pi$, the poset $\bar{\QQ}_1$ has the $\kappa$-c.c. in $M$. In particular, the antichain $A$ is of size less than $\kappa$ in $M$. 
Since $\QQ_1$ is on $\lambda$, the poset $\bar{\QQ}_1$ is on $\kappa$ in $M$. So the antichain $A$ is a bounded subset of $\kappa$. Since $\kappa$ is the critical point of $\pi$, we have that $\pi (A) = A$. 

By elementarity of $\pi$, the antichain $\pi (A) = A$ is maximal in $\QQ_1$ in $\HOD_{\{ T \}} [G]$.
Since $M$ is in $\HOD_{\{ T \}}$ and $A$ is in $M$, the antichain $A$ is maximal in $\QQ_1$ in $\HOD_{\{ T \}}$ as well. 
By the genericity of $h_x$ over $\HOD_{\{ T \}}$, the set $A \cap h_x$ is nonempty. Let $q$ be an element of $A \cap h_x$. 
Since $A$ is in $M$ and $M$ is transitive, the condition $q$ is in $M$. 
Since $\QQ_1$ is on $\lambda$, $\pi (\kappa) = \lambda$, and $q \in h_x \cap M$, we have that $\pi (q) = q$ and hence $q \in \bar{h}_x$.
Therefore, $q \in A \cap \bar{h}_x$ and the set $A \cap \bar{h}_x$ is nonempty, as desired.

Since the poset $\bar{\QQ}_1$ has the $\kappa$-c.c. in $M$, by a standard argument, one can lift the embedding $\pi \colon M \to N$ to an elementary embedding $\hat{\pi} \colon M[\bar{h}_x] \to N [h_x]$ such that $\hat{\pi} (\bar{h}_x) = h_x$. 

We now argue that the function $f$ is in $V$. It is enough to verify that $f$ is in $M [\bar{h}_x]$ because $M$ is in $\HOD_{\{ T \}}$, $\HOD_{\{ T \}} \subseteq V$, and $\bar{h}_x = \{ \bar{q} \mid q \in h_x \cap X \} = \{ q \mid q \in h_x \cap X \} = \bar{\QQ}_1 \cap h_x$. 
Recall that $\tau$ is a $\QQ_1$-name in $\HOD_{\{ T \}}$ such that $\tau^{h_x} = \dot{f}$ and that $\dot{f}$ is a $\PP$-name in $\HOD_{\{ T \}} [h_x]$ such that $\dot{f}^G = f$. 
Since $\tau$ is in $X$, letting $\dot{g} = \bar{\tau}^{\bar{h}_x}$ and $g = \dot{g}^{\bar{G}}$, we have that $\hat{\pi} (g) = f$. By elementarity of $\hat{\pi}$,  the set $g$ is a function from $\pi^{-1} (\gamma)$ to $2$. 
We now verify that $f=g$, which would imply that $f$ is in $M[\bar{h}_x]$ because $g$ is in $M[\bar{h}_x]$.
Since $\gamma + 1 \subseteq X$, we have that $\pi \upharpoonright (\gamma + 1) = \id$. 
Hence $\pi^{-1} (\gamma) = \gamma$ and $g \colon \gamma \to 2$. Also, since $\hat{\pi} (g) = f$, by elementarity of $\hat{\pi}$, for any $\alpha < \gamma$ and $i < 2$, we have that $g (\alpha) = i$ if and only if $f( \alpha) = i$. Therefore, $f = g $, as desired. 




This completes the proof of Claim~\ref{no-bounded-subset-in-vg-non-adr}.
\end{proof}

By Claim~\ref{no-bounded-subset-in-vg-non-adr}, we know that $\RR^V = \RR^{V[G]}$. So we simply write $\RR$ for $\RR^V$ or $\RR^{V[G]}$.
Recall that we write $\lambda$ for $\Theta^V$.

We now show that $\PP$ does not add any set of reals to $V$:
\begin{claim}\label{no-set-of-reals-in-vg-non-adr}
The equality $\wp (\RR)^V = \wp (\RR)^{V[G]}$ holds.
\end{claim}

\begin{proof}[Proof of Claim~\ref{no-set-of-reals-in-vg-non-adr}]

Let $A$ be any set of reals in $V[G]$. We will show that $A$ is in $V$ as well.

We first claim that $\lambda$ is regular in $V[G]$ and $\lambda = \Theta^{V[G]}$. 
Let $\QQ_{\omega}$ be the finite support direct limit of Vop\v{e}nka algebras for adding an element of $2^{\omega}$ in $\HOD_{\{ T \}}$. 
Then by Lemma~\ref{lem:Qomega}, the poset $\QQ_{\omega}$ has the $\lambda$-c.c. in $\HOD_{\{ T \}}$ and there is a $\QQ_{\omega}$-generic filter $H$ over $\HOD_{\{ T \}}$ such that $V = \LL (T, \RR) \subseteq \HOD_{\{ T \}} [H]$ and the set $\RR$ is countable in $\HOD_{\{ T \}} [H]$. 
Since $\PP$ is $<$$\lambda$-closed in $\HOD_{\{ T \}}$ and $\QQ_{\omega}$ has the $\lambda$-c.c. in $\HOD_{\{ T \}}$, by Lemma~\ref{lem:useful}, the poset $\QQ_{\omega}$ still has the $\lambda$-c.c. in $\HOD_{\{ T \}} [G]$ and the filter $H$ is $\QQ_{\omega}$-generic over $\HOD_{\{ T \}} [G]$ as well. 
Hence $\lambda$ is still regular uncountable in $\HOD_{\{ T \}} [G][H]$, the filter $G \times H$ is $\PP \times \QQ_{\omega}$-generic over $\HOD_{\{ T \}}$, and $\HOD_{\{ T \}} [G][H] = \HOD_{\{ T \}} [H][G]$. Therefore, $\lambda$ is still regular uncountable in $\HOD_{\{ T \}} [H][G]$. Since $V[G] \subseteq \HOD_{\{ T \}} [H][G]$, the ordinal $\lambda$ is regular in $V[G]$ as well.
Also, since $\RR$ is countable in $\HOD_{\{ T \}} [H]$ and $\RR^V = \RR^{V[G]}$ while $V[G] \subseteq \HOD_{\{ T \}} [H][G]$, the ordinal $\Theta^{V[G]}$ is at most $\omega_1$ in $\HOD_{\{ T \}}[H][G]$.
Since $\lambda$ is regular uncountable in $\HOD_{\{ T \}} [H][G]$, we have that $\Theta^{V[G]} \le \lambda$. Since $V \subseteq V[G]$ and $\lambda = \Theta^V$, the inequality $\lambda \le \Theta^{V[G]} $ also holds. Hence $\lambda = \Theta^{V[G]}$, as desired.

Let $\nu$ be a sufficiently large cardinal in $V[G]$ and let $N$ be $V_{\nu}$ in $V[G]$. Since $V = \LL (T, \RR)$, the model $N$ is of the form $\LL_{\nu} (T, \RR) [G]$. Since every element of $N$ is definable from $T, G$, an ordinal, and some real while $\lambda$ is regular in $V[G]$ and $\lambda = \Theta^{V[G]}$, one can find an elementary substructure $X$ of $N$ in $V[G]$ such that $\RR \subseteq X$, $\lambda \cap X \in \lambda$, the structure $X$ is a surjective image of $\RR$, and $T, \PP , G,  A \in X$. 
Let $M$ be the transitive collapse of $X$ and let $\pi \colon M \to X$ be the inverse of the collapsing map. 
Then letting $\kappa = \lambda \cap X$, the critical point of $\pi$ is $\kappa$ and $\pi (\kappa) = \lambda$. 
For any $a$ in $X$, we write $\bar{a}$ for $\pi^{-1} (a)$, i.e., $\pi ( \bar {a}) = a$.

We will finish arguing that the set $A$ is in $V$. 
Since $\RR$ is contained in $M$ and $\pi (\bar{A}) = A$, we have $\bar{A} = A$ and the set $A$ is in $M$.
Hence it is enough to verify that the model $M$ is in $V$. 
Recall that $g = \bigcup G$ and $g \colon \lambda \to 2$. Since $G$ is simply definable from $g$, we have that $N = \LL_{\nu} (T, \RR) [G] = \LL_{\nu} (T, \RR) [g]$. 
Since $N$ is of the form $\LL_{\nu} (T, \RR) [g]$, $X$ is a surjective image of $\RR$ in $V[G]$, and $\lambda = \Theta^{V[G]}$, it follows that $M$ is of the form $\LL_{\mu} (\bar{T}, \RR) [\bar{g}]$ for some ordinal $\mu < \lambda$. 
Since $\mu < \lambda$, the set $\bar{T}$ is a bounded subset of $\lambda$ in $V[G]$. 
By Claim~\ref{no-bounded-subset-in-vg-non-adr}, the set $\bar{T}$ is in $V$ as well.
Since $g \colon \lambda \to 2$ and $\pi (\kappa) = \lambda$, we have that $\bar{g} = g \upharpoonright \kappa$ and $\bar{g}$ is in $\PP$. So $\bar{g}$ is in $V$ as well.
Since $M = \LL_{\mu} (\bar{T}, \RR) [\bar{g}]$, the model $M$ is in $V$, and the set $A$ is in $V$, as desired.


This completes the proof of Claim~\ref{no-set-of-reals-in-vg-non-adr}.
\end{proof}

By Claim~\ref{no-bounded-subset-in-vg-non-adr} and Claim~\ref{no-set-of-reals-in-vg-non-adr}, we have that $\RR^V = \RR^{V[G]}$ and $\wp (\RR)^V = \wp (\RR)^{V[G]}$. 
Since we assume $\AD$ in $V$, the axiom $\AD$ holds in $V[G]$ as well.

This finishes the arguments for Theorem~\ref{thm:subset-of-Theta-positive} in Case~\ref{case:ADRfail} when $\AD_{\RR}$ fails.

\begin{case}\label{case:ADR}
When $\mathsf{AD}_{\mathbb{R}}$ holds.
\end{case}

Recall that we write $\lambda$ for $\Theta^V$.
Let $\mathbb{P}$ be $\Add  (\lambda , 1)$ in HOD, where $\Add  (\lambda , 1) = \{ p \mid p \colon \gamma \to 2 = \{ 0,1 \} \text{ for some $\gamma < \lambda$}\}$. Since $\mathbb{P}$ is computed in HOD and $\lambda = \Theta^V$ is inaccessible in HOD, the set $\mathbb{P}$ can be considered as a poset on $\lambda$.
Let $G$ be a $\mathbb{P}$-generic filter over $V$. We will show that $\mathsf{AD}$ holds in $V[G]$. 

\begin{claim}\label{no-new-set-of-reals-adr}
The forcing $\mathbb{P}$ does not add any new set of reals, i.e., $\wp(\mathbb{R})^{V} = \wp(\mathbb{R})^{V[G]}$.
\end{claim}

\begin{proof}[Proof of Claim~\ref{no-new-set-of-reals-adr}]
We will show that for any $f \colon \mathbb{R}^V \to 2$ in $V[G]$, the function $f$ is also in $V$. 
Since any real in $V[G]$ can be simply coded as a subset of $\RR^V$ in $V[G]$, this will show that $\RR^V = \RR^{V[G]}$ and $\wp (\RR)^V = \wp (\RR)^{V[G]}$ as well.

From now on, we write $\RR$ for $\RR^V$.  
\begin{subclaim}\label{subclaim}
For some sequence $s \in \lambda^{\omega}$, the function $f$ is in $\HOD_{\{ s\}} (\RR) [G]$.
\end{subclaim}
 
\begin{proof}[Proof of Subclaim~\ref{subclaim}]
Let $\dot{f}$ be a $\mathbb{P}$-name with $\dot{f}^G= f$. 
Since $\PP$ can be considered as a poset on $\lambda$, we may assume that $\dot{f}$ can be considered as a subset of $\lambda \times \RR \times 2$. To make it simpler, we regard $\dot{f}$ as a subset of $\lambda \times \RR$. 

Since $V= \LL \bigl( \wp (\RR) \bigr)$, there is a set $A$ of reals such that $\dot{f}$ is OD from $A$. 
For each $\alpha < \lambda$, let $X_{\alpha} = \{ x \in \RR \mid  (\alpha , x) \in \dot{f} \}$ and let $\xi_{\alpha}$ be the least ordinal $\xi < \lambda$ such that $X_{\xi}= X_{\alpha}$. For $\alpha , \beta < \lambda$, we write $\alpha \preceq \beta$ if $\xi_{\alpha} \le \xi_{\beta}$. Then the structure $(\lambda , \preceq)$ is a prewellordering. Let $\pi \colon (\lambda , \preceq ) \to (\gamma , \le)$ be the Mostowski collapsing map. For each $\delta < \gamma$, let $\eta_{\delta} = \min \pi^{-1} (\delta)$ and $Y_{\delta} = X_{\eta_{\delta} }$. Set $Y = ( Y_{\delta} \mid \delta < \gamma)$.  

Since $\dot{f}$ is OD from $A$, so is $\pi$. Also $\pi$ is essentially a set of ordinals, so $\pi$ is in $\HOD_{\{ A \}}$.  

We next verify that there is a set $B$ of reals such that $Y$ is in $\LL (B, \RR)$. 
Since we have $\AD_{\RR}$ in Case~\ref{case:ADR}, by Theorem~\ref{thm:Solovay}, there is a set $B_0$ of reals which is not OD from $A$ and any real. 
Since $\dot{f}$ is OD from $A$, so is $Y$. So each set $Y_{\delta}$ of reals is OD from $A$. 
By the Wadge Lemma under  $\ZF + \AD$, each $Y_{\delta}$ is Wadge reducible to $B_0$. 
In particular, there is a surjection from $\RR$ to $\{ Y_{\delta} \mid \delta < \gamma \}$ in $\LL (B_0, \RR)$. Since the sequence $Y = ( Y_{\delta} \mid \delta < \gamma )$ is injective, there is a surjection $\rho \colon \RR \to \gamma$ in $\LL (B_0, \RR)$ as well. 
Let $B_1 = \{ x \ast y \mid x \in Y_{\rho (y)} \}$, where $x \ast y \, (2n) = x(n)$ and $x \ast y \, (2n+1) = y (n)$ for all $n \in \omega$. 
Then $Y$ is in $\LL (\rho , B_1, \RR)$. 
So letting $B = B_0 \oplus B_1 = \{ x \ast y \mid x \in B_0 \text { and } y \in B_1 \}$, we have that $Y$ is in $\LL (B, \RR)$, as desired.

We now argue that for some sequence $s \in \lambda^{\omega}$, the $\PP$-name $\dot{f}$ is in $\HOD_{\{ s\}} (\RR)$. 
Since $\pi$ is in $\HOD_{\{ A \}}$ and $Y$ is in $\LL (B, \RR)$, letting $C = A \oplus B$, we have that $\pi$ is in $\HOD_{\{ C \}}$ and $Y$ is in $\LL (C, \RR)$. 
Since we have $\AD_{\RR}$ in Case~\ref{case:ADR}, by Lemma~\ref{thm:ADR-countable-sequences}, there is an $s \in (\Theta^V)^{\omega} = \lambda^{\omega}$ such that $C$ is OD from $s$ and that $C$ is in $\HOD_{\{ s\}} (\RR)$. Hence both $\pi$ and $Y$ are in $\HOD_{\{ s\}} (\RR)$. 
Since $\dot{f}$ is simply definable from $\pi$ and $Y$, we have that $\dot{f}$ is in $\HOD_{\{ s\}} (\RR)$, as desired.

Since $f = \dot{f}^G$ and $\dot{f}$ is in $\HOD_{\{ s\}} (\RR)$, we have that $f$ is in $\HOD_{\{ s\}} (\RR) [G]$.

This completes the proof of Subclaim~\ref{subclaim}.
\end{proof}

Since we assume that $\lambda = \Theta^V$ is regular in $V$ and $s \in \lambda^{\omega} \cap V$, we can pick an ordinal $\gamma < \lambda$ such that $s \in \gamma^{\omega}$. 
Let $\QQ_1$ be the Vop\v{e}nka algebra for adding an element of $\gamma^{\omega}$ in $\HOD$ and let $h_s = \{ p \in \QQ_1 \mid s \in \pi_1 (p) \}$, where $\pi_1 \colon \QQ_1 \to \mathcal{O}_1$ is as in Definition~\ref{def:Vop}. Then by Lemma~\ref{lem:ADR-countable-sequences}, we have that $h_s$ is a $\QQ_1$-generic filter over $\HOD$ such that $ \HOD[h_s] = \HOD_{\{ s\}}$. 

\begin{subclaim}\label{subclaim2}
The ordinal $\lambda$ is regular in $\HOD_{\{ s\}} (\RR) [G]$ and there is no surjection from $\RR^V$ to $\lambda$ in $\HOD_{\{ s\}} (\RR) [G]$.
\end{subclaim}

\begin{proof}[Proof of Subclaim~\ref{subclaim2}]

Let $\QQ_{\omega}$ be the finite support limit of the Vop\v{e}nka algebras for adding an element of $\gamma^{\omega}$ in $\HOD$.  Since $\QQ_1$ is a complete suborder of $\QQ_{\omega}$, by Lemma~\ref{lem:ADR-countable-sequences}, there is a $\QQ_{\omega}$-generic filter $H$ over $\HOD $ such that $h_s \in \HOD[H]$ and that the set $(\gamma^{\omega})^V$ is countable in $\HOD[H]$. In particular, $\HOD_{\{ s\}} (\RR) \subseteq \HOD [h_s] (\RR) \subseteq \HOD [H]$ and $\RR^V$ is countable in $\HOD[H]$. Since we have $\AD_{\RR}$ in Case~\ref{case:ADR}, by Lemma~\ref{lem:ADR-countable-sequences}, the poset $\QQ_{\omega}$ is of size less than $\Theta^V = \lambda$. 
Since $\PP$ is $<$$\lambda$-closed in $\HOD$ and $G$ is $\PP$-generic over $\HOD$, we have that any subset of $\QQ_{\omega}$ in $\HOD [G]$ is also in $\HOD$. Hence the filter $H$ is $\QQ_{\omega}$-generic over $\HOD [G]$ as well.

We now argue that  $\lambda$ is regular in $\HOD_{\{ s\}} (\RR) [G]$. 
Since $\PP$ is $<$$\lambda$-closed in $\HOD$ and $G$ is $\PP$-generic over $\HOD$, the ordinal $\lambda$ is still regular in $\HOD[G]$. Also, since $\QQ_{\omega}$ is of size less than $\lambda$ in $\HOD$ and $H$ is $\QQ_{\omega}$-generic over $\HOD [G]$, the ordinal $\lambda$ is also regular in $\HOD [G][H]$. Since $\HOD_{\{ s\}} (\RR) [G] \subseteq \HOD [H][G] = \HOD [G][H]$, the ordinal $\lambda$ is regular in $\HOD_{\{ s\}} (\RR) [G]$, as desired.  

We next show that there is no surjection from $\RR^V$ to $\lambda$ in $\HOD_{\{ s\}} (\RR) [G]$. 
Since $\RR^V$ is countable in $\HOD[H]$ while $\lambda$ is regular uncountable in $\HOD[H][G]$, there is no surjection from $\RR^V$ to $\lambda$ in $\HOD[H][G]$. Since $\HOD_{\{ s\}} (\RR) [G] \subseteq \HOD[H][G]$, there is no surjection from $\RR^V$ to $\lambda$ in $\HOD_{\{ s\}} (\RR) [G]$, as desired.

This completes the proof of Subclaim~\ref{subclaim2}.
\end{proof}

Recall that $\QQ_1$ is the Vop\v{e}nka algebra for adding an element of $\gamma^{\omega}$ and $h_s$ is the $\QQ_1$-generic filter over $\HOD$ derived from $s$ with $\HOD [h_s] = \HOD_{\{ s\}}$. Since we have $\AD_{\RR}$ in Case~\ref{case:ADR}, by Lemma~\ref{lem:ADR-countable-sequences}, the poset $\QQ_1$ is of size less than $\Theta^V = \lambda$ in $\HOD$. So the filter $h_s$ is essentially a bounded subset of $\lambda$. 
Since we assume $\ZF + \AD^+ + \lq\lq V= \LL \bigl( \wp (\RR) \bigr)$'', by Theorem~\ref{thm:HODLA}, there is a set $Z \subseteq \Theta^V = \lambda$ such that $\HOD = \LL [Z]$. So the model $\HOD_{\{ s\}} (\RR) [G]$ is of the form $\LL (Z, h_s ,\RR) [G]$ where $Z$ is a subset of $\lambda$ and $h_s$ is a bounded subset of $\lambda$.
By Subclaim~\ref{subclaim}, the function $f$ is in the model $\HOD_{\{ s\}} (\RR) [G]$. 

We are now ready to finish the arguments for Claim~\ref{no-new-set-of-reals-adr}, which are similar to those for Claim~\ref{no-set-of-reals-in-vg-non-adr}.
Let $\nu$ be a sufficiently big cardinal in $\HOD_{\{ s\}} (\RR) [G]$ and let $N$ be $V_{\nu}$ in $V[G]$. 
Let $g = \bigcup G$. Then by the genericity of $G$, we have $g \colon \lambda \to 2$. Since $G$ is simply definable from $g$, we also have $\HOD_{\{ s \}} (\RR) [G] = \HOD_{ \{ s \}} (\RR) [g]$. 
Since $\HOD_{\{ s\}} (\RR) = \LL (Z, h_s , \RR)$, the model $N$ is of the form $\LL_{\nu} (Z , h_s, \RR) [G] = \LL_{\nu} (Z , h_s , \RR) [g]$. Since every element of $N$ is definable from $Z, h_s, g$, an ordinal, and some real while $\lambda$ is regular in $\HOD_{\{ s\}} (\RR) [G]$ and there is no surjection from $\RR$ to $\lambda$ in $\HOD_{\{ s\}} (\RR) [G]$ by Subclaim~\ref{subclaim2}, one can find an elementary substructure $X$ of $N$ in $\HOD_{\{ s\}} (\RR) [G]$ such that $\RR \subseteq X$, $\lambda \cap X \in \lambda$, the structure $X$ is a surjective image of $\RR$, and $Z, \PP , h_s, G,  f \in X$. 
Let $M$ be the transitive collapse of $X$ and let $\pi \colon M \to X$ be the inverse of the collapsing map. 
Then letting $\kappa = \lambda \cap X$, the critical point of $\pi$ is $\kappa$ and $\pi (\kappa) = \lambda$. 
For any $a$ in $X$, we write $\bar{a}$ for $\pi^{-1} (a)$, i.e., $\pi ( \bar {a}) = a$.

We will finish arguing that the function $f$ is in $V$. 
Since $\RR$ is contained in $M$ and $\pi (\bar{f}) = f$, we have $\bar{f} = f$ and the set $f$ is in $M$.
Hence it is enough to verify that the model $M$ is in $V$. 
Since $N$ is of the form $\LL_{\nu} (Z, h_s,  \RR) [g]$, 
the set $M$ is of the form $\LL_{\mu} (\bar{Z}, \bar{H}_s, \RR) [\bar{g}]$ for some ordinal $\mu$. 
Since $Z$ is a subset of $\lambda$, we have $\bar{Z} = Z \cap \kappa$, which is in $V$. 
The filter $h_s$ is essentially bounded subset of $\lambda$ and $h_s$ is in $X$. So by elementarity of $X$, we have that $h_s \subseteq X$ and it follows that $\bar{h}_s = h_s$, which is also in $V$. 
Since $g \colon \lambda \to 2$, we have $\bar{g} = g \upharpoonright \kappa$ and so $\bar{g}$ is in $\PP$. Hence we have $\bar{g} \in V$. 
Since $M = \LL_{\mu} (\bar{Z}, \bar{h}_s , \RR) [\bar{g}]$, the model $M$ is in $V$, and hence the function $f$ is in $V$, as desired.

This completes the proof of Claim~\ref{no-new-set-of-reals-adr}.
\end{proof}

By Claim~\ref{no-new-set-of-reals-adr}, we have $\wp(\mathbb{R})^{V[G]} = \wp(\mathbb{R})^{V}$. Since $\AD$ holds in $V$, so does in $V[G]$, as desired.

This finishes the arguments for Theorem~\ref{thm:subset-of-Theta-positive} in Case~\ref{case:ADR} when $\mathsf{AD}_{\mathbb{R}}$ holds.

This completes of the proof of Theorem~\ref{thm:subset-of-Theta-positive}. 
\end{proof}


\begin{proof}[Proof of Theorem~\ref{thm:subset-of-Theta-negative}]

Let $G$ be any $\PP$-generilc filter over $V$. We will show that $\AD$ fails in $V[G]$.
To derive a contradiction, we assume $\AD$ in $V[G]$.

Since we have $\ZF + \AD^+ + \lq\lq V = \LL \bigl( \wp (\RR) \bigr)$'' in $V$, by Theorem~\ref{thm:increasing-Theta-destruction}, it is enough to show that $\PP$ increases $\Theta$, i.e., $\Theta^V < \Theta^{V[G]}$. 

Let $\gamma$ be the cofinality of $\Theta$ in $V$. Since $\Theta$ is singular in $V$, we have that $\gamma < \Theta$.

We will show that there is an injection from $\Theta^V$ to $\wp (\gamma)^{V[G]}$ in $V[G]$, which would imply $\Theta^V < \Theta^{V[G]}$ as follows: Since we assumed $\AD$ in $V[G]$, by Theorem~\ref{thm:CodingLemma}, there is a surjection from $\RR^{V[G]}$ to $\wp (\gamma)^{V[G]}$ in $V[G]$. 
By the existence of an injection from $\Theta^V$ to $\wp (\gamma)^{V[G]}$, there would be a surjection from $\RR^{V[G]}$ to $\Theta^V$ in $V[G]$. By the definition of $\Theta^{V[G]}$, we would have that $\Theta^V < \Theta^{V[G]}$, as desired.

We will construct a function $\iota \colon \Theta^V \to \wp (\gamma)^{V[G]}$ in $V[G]$ which is verified to be injective. 
Since $\PP = \Add (\Theta , 1)$ in HOD and $G$ is $\PP$-generic over $V$, the set $g = \bigcup G$ is a function from $\Theta^V$ to $2 = \{ 0, 1\}$. 
Since $\gamma$ is the cofinality of $\Theta$ in $V$, we can fix a cofinal increasing sequence $(\beta_{\alpha} \colon \alpha < \gamma )$ in $\Theta$ in $V$. 
For each $\delta < \Theta^V$, let $a_{\delta}$ be the sequence $(\beta_{\alpha} + \delta \mid \alpha < \gamma)$ in $\Theta$ in $V$. Now let $\iota (\delta) = \{ \alpha < \gamma \mid g \bigl( a_{\delta} (\alpha) \bigr)  =1 \}$. Then $\iota (\delta)$ is a subset of $\gamma$ for each $\delta < \Theta^V$.

We will verify that the function $\iota \colon \Theta^V \to \wp (\gamma)^{V[G]}$ is injective.
Let $\delta , \epsilon$ be distinct ordinals less than $\Theta^V$. We will see that $\iota (\delta) \neq \iota (\epsilon)$. 
First notice that the functions $a_{\delta}$ and $a_{\epsilon}$ are different everywhere: 
For all $\alpha < \gamma$, we have $a_{\delta} (\alpha) = \beta_{\alpha} + \delta \neq \beta_{\alpha} + \epsilon = a_{\epsilon} (\alpha)$. 
Now since $a_{\delta}$ and $a_{\epsilon}$ are different everywhere and both are in $V$, by the genericity of $G$, there is an $\alpha < \gamma$ such that $g\bigl ( a_{\delta} (\alpha) \bigr) \neq g \bigl( a_{\epsilon} (\alpha) \bigr)$, and hence $\alpha \in \iota (\delta) \bigtriangleup \iota (\epsilon)$. Therefore, we have $\iota (\delta) \neq \iota (\epsilon)$, as desired.

This completes the proof of Theorem~\ref{thm:subset-of-Theta-negative}.
\end{proof}

\section{Questions}\label{sec:Q}

We close this paper by raising two questions.

\begin{Q}\label{q1}
Assume $\ZF + \AD$. Let $\PP$ be a poset which adds a new real and let $G$ be $\PP$-generic over $V$. Then must $\AD$ fail in $V[G]$?
\end{Q}

To answer \lq No' to Question~\ref{q1}, one would need to find a poset which changes the structure of cardinals below $\Theta$ drastically as follows: 
Woodin proved that if there is a poset which adds a new real while preserving the truth of $\AD$, then the poset must collapse $\omega_1$. 
Also, by Theorem~\ref{thm:increasing-Theta-destruction}, if such a poset exists in a model of $\ZF + \AD^+ + \lq\lq V = \LL \bigl( \wp (\RR) \bigr)$'', then the poset must preserve $\Theta$. 
Furthermore, by the arguments for \cite[Lemma~2.10]{MR4242147} by Chan and Jackson, if a poset adds a new real while preserving the truth of $\AD$, then any weak partition preperty of a cardinal in its generic extension cannot be witnessed by a club in the ground model. Hence, if such a poset preserves $\Theta$ as well, then for cofinaly many cardinals $\kappa$ below $\Theta$, the poset must shoot a club in $\kappa$ which does not contain any club in $\kappa$ in the ground model. 

There are many things we do not know on forcings over $\ZF + \AD$ especially when $\Theta$ is singular. One of them is whether the assumption \lq\lq $\Theta$ is regular'' in Theorem~\ref{thm:subset-of-Theta-positive} is essential or not: 
\begin{Q}
Assume $\mathsf{ZF}+\mathsf{AD}^+ + \lq\lq V = \mathrm{L} \bigl(\wp (\mathbb{R})\bigr)"$. Suppose that $\Theta$ is singular. Then is there any poset which adds a new subset of $\Theta$ while preserving $\AD$?
\end{Q}


\bibliographystyle{plain}
\bibliography{myreference}

\def\polhk#1{\setbox0=\hbox{#1}{\ooalign{\hidewidth
  \lower1.5ex\hbox{`}\hidewidth\crcr\unhbox0}}}
\begin{thebibliography}{10}

\bibitem{MR4132099}
William Chan.
\newblock An introduction to combinatorics of determinacy.
\newblock In {\em Trends in set theory}, volume 752 of {\em Contemp. Math.},
  pages 21--75. Amer. Math. Soc., 2020.

\bibitem{MR4242147}
William Chan and Stephen Jackson.
\newblock The destruction of the axiom of determinacy by forcings on
  {$\Bbb{R}$} when {$\Theta$} is regular.
\newblock {\em Israel J. Math.}, 241(1):119--138, 2021.

\bibitem{Cunningham}
Daniel~W. Cunningham.
\newblock On forcing over ${L} (\mathbb{R})$.
\newblock {\em Arch. Math. Logic}, 2023.
\newblock To appear. Available at
  \url{https://doi.org/10.1007/s00153-022-00844-4}.

\bibitem{Jech}
Thomas Jech.
\newblock {\em Set theory}.
\newblock Springer Monographs in Mathematics. Springer-Verlag, Berlin, 2003.
\newblock The third millennium edition, revised and expanded.

\bibitem{MR1994835}
Akihiro Kanamori.
\newblock {\em The higher infinite. {L}arge cardinals in set theory from their
  beginnings}.
\newblock Springer Monographs in Mathematics. Springer-Verlag, Berlin, second
  edition, 2003.

\bibitem{MR736611}
Alexander~S. Kechris.
\newblock The axiom of determinacy implies dependent choices in {$L({\bf R})$}.
\newblock {\em J. Symbolic Logic}, 49(1):161--173, 1984.

\bibitem{ADplus}
Paul~B. Larson.
\newblock Extensions of the axiom of determinacy, 2019.
\newblock A book in progress on $\AD^+$. In preparation.

\bibitem{MR2463620}
Donald~A. Martin and John~R. Steel.
\newblock The tree of a {M}oschovakis scale is homogeneous.
\newblock In {\em Games, scales, and {S}uslin cardinals. {T}he {C}abal
  {S}eminar. {V}ol. {I}}, volume~31 of {\em Lect. Notes Log.}, pages 404--420.
  Assoc. Symbol. Logic, Chicago, IL, 2008.

\bibitem{new_Moschovakis}
Yiannis~N. Moschovakis.
\newblock {\em Descriptive set theory}, volume 155 of {\em Mathematical Surveys
  and Monographs}.
\newblock American Mathematical Society, Providence, RI, second edition, 2009.

\bibitem{Solovay_AD_R}
Robert~M. Solovay.
\newblock The independence of {DC} from {AD}.
\newblock In Alexander~S. Kechris and Yiannis~N. Moschovakis, editors, {\em
  Cabal Seminar 76--77 (Proc. Caltech-UCLA Logic Sem., 1976--77)}, volume 689
  of {\em Lecture Notes in Math.}, pages 171--183. Springer, Berlin, 1978.

\bibitem{MR2463615}
John~R. Steel.
\newblock Scales in {${\bf K}(\Bbb R)$}.
\newblock In {\em Games, scales, and {S}uslin cardinals. {T}he {C}abal
  {S}eminar. {V}ol. {I}}, volume~31 of {\em Lect. Notes Log.}, pages 176--208.
  Assoc. Symbol. Logic, Chicago, IL, 2008.

\bibitem{ADplusreflection}
John~R. Steel and Nam Trang.
\newblock $\sf{AD}^+$, {D}erived models, and {$\Sigma_1$-Reflection}, 2010.
\newblock A preprint. Available at
  \url{https://math.unt.edu/~ntrang/AD+reflection.pdf}.

\bibitem{NamThesis}
Nam Trang.
\newblock {\em Generalized Solovay Measures, the HOD Analysis, and the Core
  Model Induction}.
\newblock PhD thesis, 2013.
\newblock Available at \url{https://math.unt.edu/~ntrang/thesis.pdf}.

\bibitem{Wadge_van_Wesep}
Robert Van~Wesep.
\newblock Wadge degrees and descriptive set theory.
\newblock In Alexander~S. Kechris and Yiannis~N. Moschovakis, editors, {\em
  Cabal Seminar 76--77 (Proc. Caltech-UCLA Logic Sem., 1976--77)}, volume 689
  of {\em Lecture Notes in Math.}, pages 151--170. Springer, Berlin, 1978.

\end{thebibliography}

\end{document}